\documentclass[12pt, letterpaper,reqno]{amsart}
\usepackage{amsthm,amssymb,mathrsfs}
\usepackage{amsmath,amscd}
\usepackage{graphicx} 
\usepackage{xcolor}
\usepackage{pinlabel}
\usepackage{fourier}
\usepackage[breaklinks,colorlinks,citecolor=blue,linkcolor=blue,
 urlcolor=teal]{hyperref} 

\theoremstyle{plain}
\newtheorem{theorem}[equation]{Theorem}

\newtheorem{lemma}[equation]{Lemma}
\newtheorem{corollary}[equation]{Corollary}
\newtheorem*{mainthm}{Theorem~\ref{mainthm}}
\newtheorem*{maincor}{Corollary~\ref{raagcor}}

\theoremstyle{definition}
\newtheorem{definition}[equation]{Definition}
\newtheorem{remark}[equation]{Remark}

\newcommand{\co}{\colon \thinspace}

\newcommand{\R}{{\mathbb R}}
\newcommand{\Z}{{\mathbb Z}}

\newcommand{\Hyp}{{\mathbb H}}

\newcommand{\abs}[1]{\lvert {#1} \rvert}

\renewcommand{\leq}{\leqslant}
\renewcommand{\geq}{\geqslant}
\renewcommand{\epsilon}{\varepsilon}
\renewcommand{\phi}{\varphi}
\newcommand{\N}{{\mathbb N}}

\newcommand{\Rs}{{\mathcal R}}
\newcommand{\Fs}{{\mathcal F}}

\renewcommand{\preceq}{\preccurlyeq}

\DeclareMathOperator{\id}{id}
\DeclareMathOperator{\area}{Area}

\DeclareMathOperator{\FA}{FA}
\DeclareMathOperator{\Sol}{Sol}

\numberwithin{equation}{section}

\makeatletter
  \def\tagform@#1{\maketag@@@{%
   \textbf{(\ignorespaces#1\unskip\@@italiccorr)}}}%
   \renewcommand{\eqref}[1]{\textup{\maketag@@@{(\ignorespaces%
        {\ref{#1}}\unskip\@@italiccorr)}}}
\makeatother

\begin{document}

\title[Stallings--Bieri groups]{The Dehn functions of Stallings--Bieri groups}

\author{William Carter}
\address{Mathematics Department\\
        University of Oklahoma\\
        Norman, OK 73019\\
        USA}

\author{Max Forester}
\email{wpcarter@ou.edu\\forester@math.ou.edu}

\begin{abstract} 
We show that the Stallings--Bieri groups, along with certain other
Bestvina--Brady groups, have quadratic Dehn function. 
\end{abstract}

\maketitle

\thispagestyle{empty}

\section{Introduction}

Every simplicial graph $\Gamma$ defines a \emph{right-angled Artin group}
$A_{\Gamma}$, whose generators correspond to vertices of $\Gamma$ and
where two such generators commute if and only if their vertices bound an
edge in $\Gamma$; see Section \ref{sec:prelim} for details.

Given $A_{\Gamma}$, there is a surjective homomorphism $\phi \co
A_{\Gamma} \to \Z$ sending each standard generator to $1$. The
\emph{Bestvina--Brady group} associated to $A_{\Gamma}$, denoted $B\!
B_{\Gamma}$, is defined to be the kernel of $\phi$.

In this paper, we are concerned with the isoperimetric behavior of
certain Bestvina--Brady groups. Dison \cite{Dison} has shown that every
Bestvina--Brady group satisfies a quartic isoperimetric inequality. The
Dehn function of such a group (that is, the \emph{optimal} isoperimetric
function) may be quartic \cite{ABDDY}, or it may be smaller. The examples
of primary interest to us are the \emph{Stallings--Bieri groups}. Some
time ago, Gersten \cite{Gersten} established a quintic isoperimetric
bound for these groups and inquired whether their Dehn functions might be
smaller. Bridson argued in \cite{Bridson:doubles} that these Dehn
functions should be quadratic. However, Bridson and Groves later found an
error and observed that the method only provided a cubic bound
\cite{Bridson,Groves}. In this paper, we give a proof of Bridson's claim.

For $n \geq 1$, the \emph{Stallings--Bieri group} $SB_n$ is defined to be
the Bestvina--Brady group associated with the right-angled Artin group
$F_2 \times \dotsm \times F_2$, where there are $n$ factors $F_2$. That
is, $SB_n$ is equal to $B\!  B_{\Gamma}$, where $\Gamma$ is the join of
$n$ copies of $S^0$ (the graph with two vertices and no edges).  The
groups $SB_n$ are notable for their homological finiteness
properties. Recall that a group $G$ is said to be \emph{of type $\Fs_n$}
if there is a $K(G,1)$ with finite $n$--skeleton. Bieri
\cite{Bieri:notes} has shown that $SB_n$ is of type $\Fs_{n-1}$ but not
of type $\Fs_n$.

Note that $SB_1$ and $SB_2$ are not finitely presented. 
The first of the groups $SB_n$ for which the Dehn function is defined is
$SB_3$, also known as Stallings' group \cite{Stallings}. Dison, Elder,
Riley, and Young \cite{DERY} proved that $SB_3$ has quadratic Dehn
function. Their method makes use of a particular presentation
for $SB_3$ given in \cite{BBMS,Gersten}, and is essentially algebraic in
nature. Their technique does not appear to generalize easily to the other
groups $SB_n$. 

In this paper, we approach the Dehn function of $SB_n$ from a geometric
point of view, by considering these groups as level sets in 
products of CAT(0) spaces and using the ambient CAT(0) geometry. 

The specific setting of our main result is that of cube complexes with
height functions. Our notion of height function is specific to cube
complexes and is slightly more restrictive than the combinatorial Morse
functions of \cite{BeBr}; see Section \ref{sec:prelim}. If $h \co X \to
\R$ is a height function on a cube complex $X$, we denote by $[X]_0$ the
level set $h^{-1}(0)$. Our main theorem is the following: 

\begin{mainthm}
Suppose $\alpha \geq 2$ and let $X_1$, $X_2$, and $X_3$ be simply
connected cube complexes with height functions such 
that each $X_i$ is admissible and has finite-valued Dehn function $\preceq
n^{\alpha}$. Then $[X_1 \times X_2 \times X_3]_0$ is simply connected
and has Dehn function $\preceq n^{\alpha}$. 
\end{mainthm}

Since right-angled Artin groups act naturally on CAT(0) cube complexes
with height functions, one easily obtains the following corollary. 

\begin{maincor}
Suppose $\Gamma = \Gamma_1 \ast \Gamma_2 \ast \Gamma_3$, so that
$A_{\Gamma}$ is the product $A_{\Gamma_1} \times A_{\Gamma_2} \times
  A_{\Gamma_3}$. Then the Bestvina-Brady group $B\! B_{\Gamma}$ has
  quadratic Dehn function. 
\end{maincor}

In particular, $SB_n$ has quadratic Dehn function for every $n \geq
3$. This yields new information on the class of groups having
quadratic Dehn function: 

\begin{corollary}
For each $n \geq 3$ there exist groups with quadratic Dehn function that
are of type $\Fs_{n-1}$ but not of type $\Fs_n$. 
\end{corollary}

\subsection*{Methods} 
As mentioned above, our basic viewpoint is the study of level sets in
products of CAT(0) spaces, or of more general spaces. The case of
horospheres in products has been well studied, for instance by 
Gromov \cite{Gromov} and Dru\c{t}u \cite{Drutu}, and our approach here
uses similar ideas. 

If one is looking at a product of CAT(0) spaces with a height function,
it is important to note that the zero-level set is not generally
CAT(0). However, thanks to the product structure, we will be able to find
many overlapping CAT(0) \emph{subspaces} within the zero-level set. An
example of this phenomenon, discussed in \cite[2.B(f)]{Gromov}, is the
solvable Lie group $\Sol_5$ considered as a horosphere in $X = \Hyp^2
\times \Hyp^2 \times \Hyp^2$. The height function in this case is a
Busemann function on $X$. While $\Sol_5$ is certainly not non-positively
curved, it contains many isometrically embedded copies of $\Hyp^2 \times
\Hyp^2$. Indeed, there are three transverse copies passing through every
point. 

In our setting of cube complexes with height functions, the appropriate
subspaces are found using the Embedding Lemma (\ref{embedding}). We
insist on using a particular cell structure (the \emph{sliced cell
  structure}, see Section \ref{sec:embedding}). Then the lemma produces
subcomplexes of the level set that are \emph{combinatorially 
  isomorphic} to the factors of the ambient product space. These
subcomplexes can be found in abundance, if the factor cube complexes are
``admissible'' (see Section \ref{sec:disks}). 

The next basic technique that we use comes from Gromov
\cite[5.A$_3''$]{Gromov}, in which a disk is filled using triangular
regions whose areas are controlled by a geometric series. Doing so
depends on using a particular triangulation of the disk, shown in Figure
\ref{fig:farey}. We note that Young has formalized this idea in his
notion of a \emph{template} \cite{Young}, although here we do not require
the full generality of Young's notion.  

The main body of our argument entails showing how to fill the triangular
regions of the template, using the subcomplexes of the level set
provided by the Embedding Lemma. This is achieved using the scheme shown
in Figure \ref{fig:triangle}. 

\subsection*{Acknowledgments} 
The authors are grateful to Noel Brady for many valuable discussions
related to this work. The second author was partially supported
by NSF grant DMS-1105765.

 \section{Preliminaries}\label{sec:prelim}

\subsection*{Dehn functions}
Let $X$ be a simply connected cell complex. Given a closed edge path $p
\co S^1 \to X$, the \emph{filling area} of $p$, informally, is the
minimal number of $2$--cells of $X$ that must be crossed in a
nullhomotopy of $p$ in the $2$--skeleton of $X$. One way to formalize
this notion is to use \emph{admissible maps} as in \cite{BBFS}. A map $f
\co D^2 \to X$ is \emph{admissible} if its image lies in $X^{(2)}$ and
the preimage of each open $2$--cell is a disjoint union of open disks in
$D^2$, each mapping homeomorphically to its image. The \emph{area} of $f$
is the total number of preimage disks. We define the filling area of $p$
to be 
\[ \FA(p) = \min \{\, \area(f) \mid f \text{ is an admissible map
  extending } p \, \}.\]
The \emph{Dehn function} of $X$ is the function $\delta_X \co \N \to \N
\cup \{\infty\}$ given by 
\[ \delta_X(n) = \sup \{ \, \FA(p) \mid p \text{ is a closed edge path of
  length } \leq n \, \},\]
where \emph{length} is the number of edges traversed by $p$. 

There is an equivalence relation on monotone functions $f \co \N \to
\N$, where we say that $f \simeq g$ if $f \preceq g$ and $g \preceq f$. 
Here, $f \preceq g$ means that there is a constant $C$ such that $f(n)
\leq C g(Cn + C) + Cn + C$ for all $n$. 

If $G$ is a finitely presented group, then the Dehn function of any
Cayley $2$--complex for $G$ takes values in $\N$. Moreover, any two
Cayley $2$--complexes will be quasi-isometric and their Dehn functions
will be equivalent. The equivalence class of this function is, by
definition, the Dehn function of $G$. See \cite{Bridson:word} for more
background on Dehn functions, including various alternative definitions. 

\subsection*{Right-angled Artin groups}
Given a simplicial graph $\Gamma$ with vertex set $V(\Gamma)$ and edge
set $E(\Gamma)$, the \emph{right-angled Artin group}
$A_{\Gamma}$ is the group with generating set  $\{a_v\}_{v \in
  V(\Gamma)}$ and relations $\Rs = \{[a_v, a_u] \mid \{v,u\} \in
E(\Gamma)\}$. 

Following \cite{Davis}, there is a natural model for $K(A_{\Gamma},1)$
which is a subcomplex of a torus. Let $T^{V(\Gamma)}$ denote the
product of copies of the circle, one for each vertex of $\Gamma$. For any
subset $U \subset V(\Gamma)$ let $T^U$ be the sub-torus spanned by the
circle factors corresponding to $U$. Let $K(\Gamma)$ denote the set of
subsets of $V(\Gamma)$ which span complete subgraphs of $\Gamma$. Then we
define 
\[ T_{\Gamma} \ = \bigcup_{U \in K(\Gamma)} T^U.\]
This subcomplex of $T^{V(\Gamma)}$ is aspherical and has fundamental
group $A_{\Gamma}$. It has a piecewise Euclidean cubical structure
satisfying Gromov's link condition. Thus it has non-positive curvature,
and the universal cover is a CAT(0) cube complex. 
We will denote this universal cover by ${X_{\Gamma}}$. 

For each $U \in K(\Gamma)$ the preimage of $T^U$ in
${X_{\Gamma}}$ is a disjoint union of isometrically embedded
copies of $\R^U$, which we will call \emph{$U$--flats}, or
\emph{coordinate flats}. 

\subsection*{Height functions}
Let $X$ be a cube complex. A \emph{height function} on $X$
is a continuous map $h \co X \to \R$ which is affine on each cube and
takes each edge to an interval of the form $[n, n+1]$ with $n \in \Z$. 

If $X_1, \dotsc, X_n$ are cube complexes with height functions
$h_i$ on $X_i$, then 
\begin{equation}\label{productheight}
h(x_1,\dotsc, x_n) \ = \ h_1(x_1) + \dotsm + h_n(x_n)
\end{equation}
defines a height function on $X_1 \times \dotsm \times X_n$. Unless
stated otherwise, a product of cube complexes with height functions will
be given this height function by default. 

If $h \co X \to \R$ is a height function, we denote by $[X]_0$ the level
set $h^{-1}(0)$.  

\medskip

In the case of ${X_{\Gamma}}$, a height function can be defined as
follows. Choose a base vertex in ${X_{\Gamma}}$.  
Consider the linear map $\R^{V(\Gamma)} \to \R$ which takes each standard
basis vector to $1 \in \R$. This map descends to a map $T^{V(\Gamma)} \to
S^1$, which restricts to a map $T_{\Gamma} \to S^1$. This latter map
induces the homomorphism $\phi \co A_{\Gamma} \to \Z$ sending each
generator $a_v$ to $1$. The desired height function 
\[ h_{\Gamma} \co {X_{\Gamma}} \to \R\]
is the unique lift of the map $T_{\Gamma} \to S^1$ which takes the base
vertex of ${X_{\Gamma}}$ to $0$. Moreover, this height function
is $\phi$--equivariant. For more details on $h_{\Gamma}$ and
$\phi$, see \cite[5.12]{BeBr}. 

\begin{remark}
If $\Gamma$ is an $n$--fold join $\Gamma_1 \ast \dotsm \ast \Gamma_n$, then
$A_{\Gamma} = A_{\Gamma_1} \times \dotsm \times A_{\Gamma_n}$ and
${X_{\Gamma}}$ is the product cube complex ${X_{\Gamma_1}} \times \dotsm
\times {X_{\Gamma_n}}$. Choosing basepoints in each $X_{\Gamma_i}$
defines height functions $h_{\Gamma_i}$. Using the product
basepoint, $h_{\Gamma}$ then agrees with the height function
\eqref{productheight} built from the functions $h_{\Gamma_i}$.  
\end{remark}

\section{The embedding lemma}\label{sec:embedding}

\subsection*{The sliced cell structure} 
Let $X$ be a cube complex with a height function $h$. The
\emph{sliced cell structure} on $X$ is obtained by subdividing each cube
of $X$ along the hyperplanes $h^{-1}(n)$ for each $n \in \Z$. Each
$d$--dimensional cube is split into $d$ convex polytopes of
dimension $d$, which are affinely equivalent to \emph{hypersimplices}
(see Remark \ref{tessellation} below). 

There are two types of cells in the sliced cell
structure. \emph{Horizontal cells} are those whose image under $h$ is a
point. The rest are \emph{transverse cells}; each of these is a
piece of a cube of the same dimension, and maps to an interval $[n, n+1]$ 
under $h$. 

Whenever we have a cube complex $X$ with a height function, we will
assume that $X$ has been given the sliced cell structure, unless stated
otherwise. We may refer to it as a \emph{sliced cube complex} to
emphasize this assumption. 

Note that the level set $[X]_0$ is a subcomplex of $X$ with this
structure. 

\begin{lemma}\label{isom}
Let $X$ be a sliced cube complex with height function $h$. Give $\R$ the
structure of a cube complex with vertices at the integers. Then $X \times
\R$ is a cube complex with height function $H(x,t) = h(x) + t$. Define
the function $f \co X \to X \times \R$ by $f(x) = (x, -h(x))$. Then $f$
is a combinatorial isomorphism of\/ $X$ onto the subcomplex $[X \times
\R]_0$. 
\end{lemma}

\begin{proof}
The claim that $H$ is a height function follows from
\eqref{productheight}, since the identity is a height function on $\R$.  

For the main conclusion, 
it is clear that $f$ is a homeomorphism from $X$ to $[X \times
\R]_0$, with inverse given by projection onto the first factor. It 
remains to show that each $d$--cell in $X$ maps bijectively to a
$d$--cell in $[X \times \R]_0$. We will show that this holds for each
\emph{transverse} $d$--cell, and moreover that the polyhedral structure
of the cell is preserved. Then, since horizontal cells are faces of
transverse cells, it follows that horizontal cells also map as desired. 

Let $\sigma$ be a transverse $d$--cell contained in a $d$--dimensional
cube $C \subset X$. There is a parametrization of $C$ as $[0,1]^d$ such
that $h\vert_C$ is given by $h(x_1, \dotsc, x_d) = x_1 + \dotsm + x_d +
N$ for some $N \in \Z$. Then $\sigma$ is defined by the inequalities
\begin{equation}\label{sigmaregion}
0 \ \leq \ x_i \ \leq \ 1 \ \ (i =  1, \dotsc, d), \quad k \ \leq \ x_1 +
\dotsm + x_d \ \leq \ k+1
\end{equation}
for some $k \in \{0, \dotsc, d-1\}$. 

The image $f(\sigma)$ lies in the set $C \times \R$ with coordinates
$x_1, \dotsc, x_d, t$. On the cube $C$, $f$ is given by 
\begin{equation}\label{affinemap}
(x_1, \dotsc, x_d) \mapsto (x_1, \dotsc, x_d, -x_1 - \dotsm - x_d -
N).
\end{equation}
Under this map, the region \eqref{sigmaregion} maps onto the region 
\begin{gather*}
0 \ \leq \ x_i \ \leq \ 1 \ \ (i =  1, \dotsc, d), \quad -k-N-1 
\ \leq \ t \ \leq \ -k-N, \\
x_1 + \dotsm + x_d + N + t = 0.
\end{gather*}
But this is simply the $0$--level set of the cube $[0,1]^d \times
[\ell, \ell+1]$, for $\ell = -k-N-1$, with respect to the height function
on $X \times \R$. That is, $f(\sigma)$ is a horizontal $d$--cell of $X
\times \R$ at height $0$. 

Finally, note that the description \eqref{affinemap} of $f$ shows that
$f\vert_C \co C \to C \times \R$ is the restriction of an injective
affine linear map $\R^d \to \R^{d+1}$, and such a map will preserve the
combinatorial structure of any convex polyhedron. 
\end{proof}

\begin{remark}\label{tessellation}
In the case $X = \R^n$ with its standard cubical structure and 
height function $h(x_1, \dotsc, x_n) = \sum_i x_i$, the
image $[\R^n \times \R]_0$ is $\R^n$ tessellated by hypersimplices. See
\cite[Section 3.3]{AR} for a description of this tessellation and its
cells. The map $f \co \R^n \to [\R^n \times \R]_0$ is affine linear. 
\end{remark}

\begin{definition}
Let $X$ be a cube complex with height function $h$. A
\emph{monotone line} in $X$ is a $1$--dimensional subcomplex $L \subset
X$ such that $h \vert_L \co L \to \R$ is a homeomorphism. 
\end{definition}

\begin{lemma}[Embedding Lemma]\label{embedding}
Let $X$ and $Y$ be sliced cube complexes with height functions $h_X,
h_Y$, and 
let $L \subset Y$ be a monotone line. The function \[f_L(x) \ = \ (x,
(h_Y\vert_L)^{-1}(-h_X(x)))\] is a combinatorial embedding of\/ $X$ into
$[X\times Y]_0$, with image $[X\times L]_0$. 
\end{lemma}

In particular, $[X \times L]_0$ is combinatorially isomorphic to $X$. 

\begin{proof}
The map $f_L$ is the composition of the combinatorial embedding $f \co X
\to X \times \R$ with image $[X \times \R]_0$ given by Lemma \ref{isom},
and the height-preserving combinatorial embedding $X \times \R \to X
\times Y$ given by $\id \times (h_Y\vert_L)^{-1}$. 
\end{proof}

\section{Filling disks in level sets}\label{sec:disks}

\begin{definition}
A cube complex $X$ with a height function is \emph{admissible} if every
vertex of $X$ is contained in a monotone line. 
\end{definition}

For any right-angled Artin group $A_{\Gamma}$, the cube complex
$X_{\Gamma}$ is admissible. Pick any vertex $v \in V(\Gamma)$, and note
that every vertex of $X_{\Gamma}$ has a $\{v\}$--flat passing through it,
and such a coordinate flat will be a monotone line for the height
function $h_{\Gamma}$. 

The following result is our main theorem. 

\begin{theorem}\label{mainthm}
Suppose $\alpha \geq 2$ and let $X_1$, $X_2$, and $X_3$ be simply
connected cube complexes with height functions such
that each $X_i$ is admissible and has finite-valued Dehn function $\preceq
n^{\alpha}$. Then $[X_1 \times X_2 \times X_3]_0$ is simply connected
and has Dehn function $\preceq n^{\alpha}$. 
\end{theorem}

\begin{corollary}\label{raagcor}
Suppose $\Gamma = \Gamma_1 \ast \Gamma_2 \ast \Gamma_3$, so that
$A_{\Gamma}$ is the product $A_{\Gamma_1} \times A_{\Gamma_2} \times
  A_{\Gamma_3}$. Then the Bestvina-Brady group $B\! B_{\Gamma}$ has
  quadratic Dehn function. 
\end{corollary}

In particular, $SB_n$ has quadratic Dehn function for every $n \geq 3$,
since it equals $B\! B_{\Gamma}$ where $\Gamma$ is the join of $n$ copies of
$S^0$. 

\begin{proof}
We have $X_{\Gamma} = X_{\Gamma_1} \times X_{\Gamma_2} \times
X_{\Gamma_3}$ where each $X_{\Gamma_i}$ is CAT(0), and therefore has Dehn
function which is at most quadratic, by \cite[Proposition
III.$\Gamma$.1.6]{BH}. Theorem \ref{mainthm} then says that 
the Dehn function of $[X_{\Gamma}]_0$ is at most quadratic. It is at
least quadratic because it contains $2$--dimensional quasi-flats, namely
the zero-level sets of any product of three monotone lines in the
factors. Finally, note that $[X_{\Gamma}]_0$ is a geometric model for
$B\! B_{\Gamma}$, as follows. It is simply connected, by Theorem
\ref{mainthm}, and is acted on freely by $B\! B_{\Gamma}$, with
quotient a finite cell complex. Hence its Dehn function is the Dehn
function of $B\! B_{\Gamma}$. 
\end{proof}

For the rest of this section, let $X_1$, $X_2$, and $X_3$ be as in the
statement of the theorem. These cube complexes will be left
\emph{unsliced}, for the purpose of estimating distances, though products
will always be given the sliced cell structure. 

Let $X = X_1 \times X_2 \times X_3$ and let $d_0( \, \cdot \, , \, \cdot
\, )$ be the combinatorial metric on the $1$--skeleton of $[X]_0$. This
is the path metric obtained by declaring each edge to be isometric to an
interval of length $1$. Let $d_{X_i}( \, \cdot \, , \, \cdot \, )$ be the
combinatorial metric on the $1$--skeleton of the (unsliced) cube
complex $X_i$.  

Consider for a moment the $1$--skeleton of the unsliced cube
complex $X_1 \times X_2 \times X_3$. Its combinatorial metric is given
by $d(a,b) = d_{X_1}(a_1, b_1) + d_{X_2}(a_2, b_2) + d_{X_3}(a_3, b_3)$
where $a = (a_1, a_2, a_3)$ and $b = (b_1, b_2, b_3)$. Given an 
edge path in $[X]_0$, every edge in the path can be replaced by a path of
length $2$ in $X_1 \times X_2 \times X_3$, which yields the following
inequality: 
\begin{equation}\label{metrics}
d_{X_1}(a_1, b_1) + d_{X_2}(a_2, b_2) + d_{X_3}(a_3, b_3) \ \leq \
2d_0(a,b). 
\end{equation}

\subsection*{Spanning triangles} 
Here we give a construction of a triangular loop in $[X]_0$ and a filling
of that loop by a topological disk in $[X]_0$. The starting data are:
three vertices $a = (a_1, a_2, a_3)$, $b = (b_1, b_2, b_3)$, and $c =
(c_1, c_2, c_3)$ in $[X]_0$ which will be the corners of the triangle,
and three monotone lines $L_i \subset X_i$ ($i = 1, 2, 3$) such that $a_1
\in L_1$, $b_2 \in L_2$, and $c_3 \in L_3$. 

Figure \ref{fig:triangle} shows the triangular loop and some vertices and
paths which will form part of the $1$--skeleton of the filling disk. In
addition to the original three corner vertices, there are six ``side
vertices'' and three interior vertices. Their coordinates in $X$ are as
indicated in the figure. The side vertices have the property that two of
their coordinates are points lying in the monotone lines. The interior
vertices have all three coordinates lying in the monotone lines. 

\begin{figure}[ht]
\labellist
\hair 3pt
\small
\pinlabel {$(\textcolor{blue}{a_1},a_2,a_3)$} [l] at 130 234
\pinlabel {$(b_1, \textcolor{blue}{b_2}, b_3)$} [r] at 0 18
\pinlabel {$(c_1, c_2, \textcolor{blue}{c_3})$} [l] at 254 18

\pinlabel {$(\, \textcolor{blue}{\circ} \, , a_2,
  \textcolor{blue}{c_3})$} [l] at 169 167
\pinlabel {$(\textcolor{blue}{a_1} , c_2, \, \textcolor{blue}{\circ} \,
  )$} [l] at 215 86 

\pinlabel {$(\, \textcolor{blue}{\circ} \, , \textcolor{blue}{b_2},
  a_3)$} [r] at 85 167
\pinlabel {$(\textcolor{blue}{a_1} , \, \textcolor{blue}{\circ} \, , b_3)$}
[r] at 39 86 

\pinlabel {$(b_1, \, \textcolor{blue}{\circ} \, , \textcolor{blue}{c_3})$}
[t] at 80 16
\pinlabel {$(c_1, \textcolor{blue}{b_2}, \, \textcolor{blue}{\circ} \,
  )$} [t] at 174 16

\pinlabel {$( \, \textcolor{blue}{\circ} \, , \textcolor{blue}{b_2},
  \textcolor{blue}{c_3})$} [l] at 194 145 
\pinlabel {$(\textcolor{blue}{a_1}, \textcolor{blue}{b_2},
   \, \textcolor{blue}{\circ} \, )$} [l] at 240 60
\pinlabel {$(\textcolor{blue}{a_1}, \, \textcolor{blue}{\circ} \, , 
  \textcolor{blue}{c_3})$} [r] at 15 60 

\tiny
\pinlabel* {$[L_1 \times X_2 \times X_3]_0$} at 128 172
\pinlabel* {\rotatebox[origin=c]{60}{$[L_1 \times L_2 \times X_3]_0$}} at
85 116 
\pinlabel* {\rotatebox[origin=c]{300}{$[L_1 \times X_2 \times L_3]_0$}}
at 170.5 116 
\pinlabel* {$[L_1 \times L_2 \times L_3]_0$} at 128 80
\pinlabel* {$[X_1 \times L_2 \times X_3]_0$} at 48 38
\pinlabel* {$[X_1 \times L_2 \times L_3]_0$} at 128 38
\pinlabel* {$[X_1 \times X_2 \times L_3]_0$} at 207 38
\endlabellist
\includegraphics[width=3in]{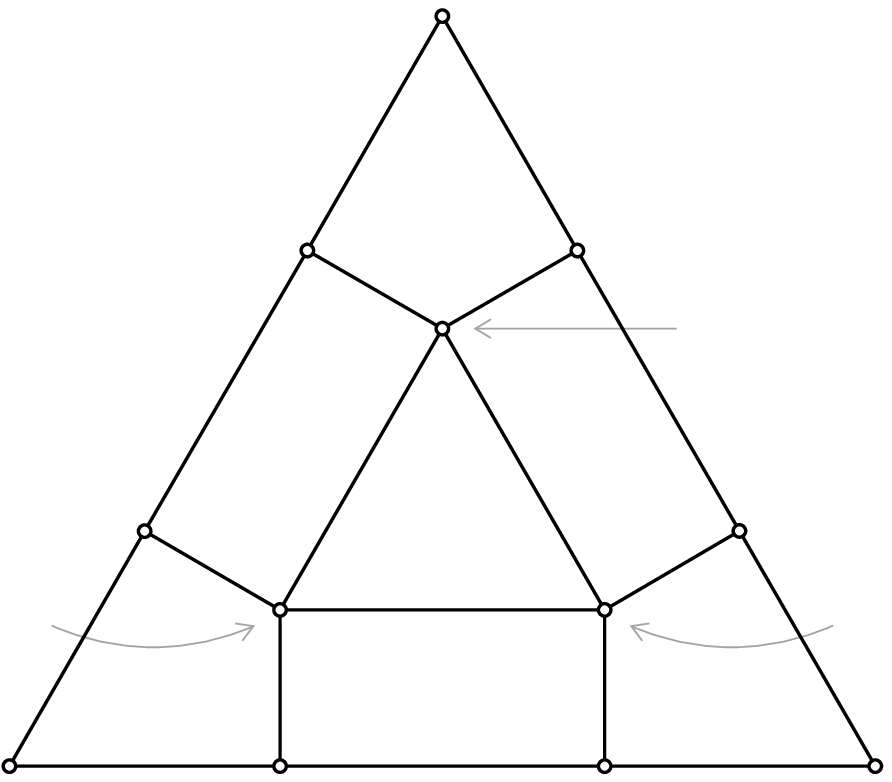}
\caption{A spanning triangle for the vertices $a, b, c \in [X]_0$. A
  blue entry indicates a point which is known to lie in the monotone line
  in that factor. A blue circle  represents the unique point in the
  monotone line for which the triple has height $0$. The
  labels in the regions indicate that each path in the boundary of that
  region lies within the indicated subspace.} 
\label{fig:triangle} 
\end{figure}

There are four types of paths in this figure: six paths joining corner
vertices to side vertices, three paths between adjacent side vertices,
six paths from side vertices to interior vertices, and three interior
paths. Consider first the path between $(a_1, a_2, a_3)$ and $(\, \circ \, ,
b_2, a_3)$ in the upper left part of the figure. We may write the second
vertex as $(a'_1, b_2, a_3)$ where $a'_1$ is the unique point
on $L_1$ such that the triple has height $0$. There is an edge path in
$X_2$ from $a_2$ to $b_2$ of length $d_{X_2}(a_2, b_2)$. Combine this
with the constant path $a_3$ in $X_3$, and the path in $L_1$ from $a_1$
to $a'_1$ which compensates for the height changes in $X_2$, keeping the
path in $[X]_0$. Put another way, this path in $[X]_0$ is the image of the
path in $X_2 \times X_3$ under the identification with $[L_1 \times X_2
\times X_3]_0$ given by the Embedding Lemma. 

Next consider the path from $(a'_1, b_2, a_3)$ to the interior vertex
$(\, \circ \, , b_2, c_3) = (a''_1, b_2, c_3)$. This is defined by
combining the constant path $b_2$ in $L_2$ with an edge path in $X_3$ from
$a_3$ to $c_3$ of length $d_{X_3}(a_3, c_3)$, and interpreting as a path in
$[L_1 \times L_2 \times X_3]_0$ via the isomorphism with $L_2 \times
X_3$. 

To define the path in $[L_1 \times L_2 \times X_3]_0$ from $(a'_1, b_2,
a_3)$ to $(a_1, \, \circ \, , b_3) = (a_1, b'_2, b_3)$ we proceed in a
somewhat non-canonical manner. First use a path in $X_3$ from $a_3$
to $b_3$, of length $d_{X_3}(a_3, b_3)$, combined with the constant path
$b_2$ in $L_2$, and a compensating motion in $L_1$, ending at a point
$(a'''_1, b_2, b_3)$. Then $(a_1, b'_2)$ and $(a'''_1, b_2)$ have the same
height in $L_1 \times L_2$, so we can move horizontally from one to the
other, in $d_{L_2}(b'_2, b_2)$ steps. This path combines with the
constant path $b_3$ in $X_3$ to give a path in $[X]_0$. Note that
$d_{L_2}(b'_2,b_2) \leq d_{X_1}(a_1, b_1)$ since the path from $(a_1,
b'_2, b_3)$ to $(b_1, b_2, b_3)$ (of length $d_{X_1}(a_1, b_1)$) projects
to a path in $X_2$ from $b'_2$ to $b_2$. 

Lastly consider a path in $[L_1 \times L_2 \times L_3]_0$ between the
interior vertices $(a''_1, b_2, c_3)$ and $(a_1, \, \circ \, , c_3) =
(a_1, b''_2, c_3)$. Use a path in $L_1 \times L_2$ from $(a''_1, b_2)$
to $(a_1, b''_2)$ of length at most $d_{X_1}(a_1, a''_1) + d_{X_2}(b_2,
b''_2)$, and embed as a path in $[L_1 \times L_2 \times L_3]_0$. We have
seen already that $d_{X_1}(a_1, a''_1) \leq d_{X_2}(a_2, b_2) +
d_{X_3}(a_3, c_3)$ and $d_{X_2}(b_2, b''_2) \leq d_{X_1}(a_1, b_1) +
d_{X_3}(b_3, c_3)$. 

The rest of the paths are defined analogously according to their
types. The length information just discussed is collected in Figure 
\ref{fig:lengths}. We use the shorthand $\abs{a_i - b_i} = d_{X_i}(a_i,
b_i)$. 

\begin{figure}[ht]
\labellist
\hair 3pt
\small
\pinlabel {$\abs{a_2 - b_2}$} [r] at 87 207
\pinlabel {$\abs{a_3 - c_3}$} [l] at 168 207
\pinlabel {$\abs{a_1 - b_1}$} [r] at 6 68
\pinlabel {$\abs{a_1 - c_1}$} [l] at 249 68
\pinlabel {$\abs{b_3 - c_3}$} [r] at 10 5
\pinlabel {$\abs{b_2 - c_2}$} [l] at 245 5

\Small
\pinlabel {$\abs{a_1 - b_1} + \abs{b_3 - c_3}$} [br] at 65 160
\pinlabel {$+ \, \abs{a_2 - b_2} +  \abs{a_3 - c_3}$} [br] at 65 144

\pinlabel {${\abs{a_3 - b_3} + \abs{a_1 - b_1}}$} [r] at 54 118

\pinlabel {$\abs{a_3 - c_3} + \abs{a_2 - b_2} \, +$} [bl] at 191 160
\pinlabel {$\abs{a_1 - c_1} + \abs{b_2 - c_2}$} [bl] at 191 144

\pinlabel {${\abs{a_2 - c_2} + \abs{a_3 - c_3}}$} [l] at 200 118

\tiny
\pinlabel {${\abs{b_2 - c_2} + \abs{a_1 - c_1}}$} [b] at 129 48
\pinlabel {${+ \, \abs{b_3 - c_3} + \abs{a_1 - b_1}}$} [t] at 126 51

\Small
\pinlabel {${\abs{b_1 - c_1} + \abs{b_2 - c_2}}$} [t] at 128 15
\endlabellist
\includegraphics[width=3in]{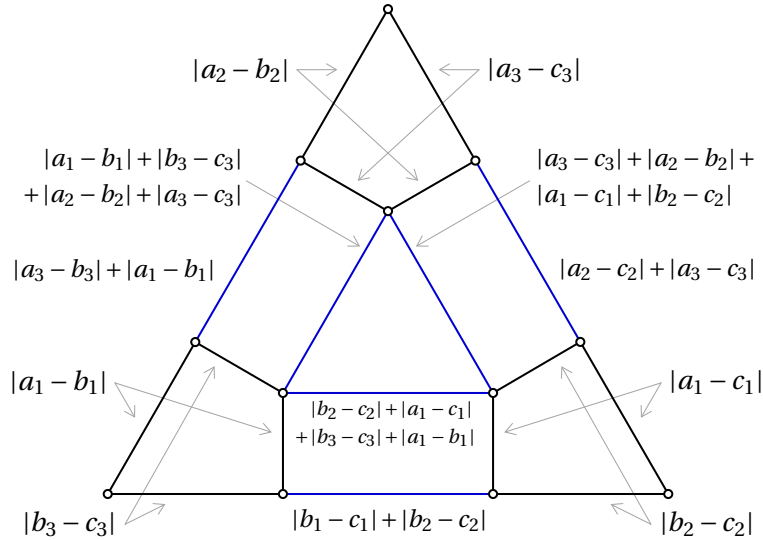}
\caption{Lengths in the spanning triangle. For the blue paths,
  the lengths given are upper bounds only.}\label{fig:lengths} 
\end{figure}

\begin{definition}
The quantity $d_0(a,b) + d_0(b,c) + d_0(c,a)$ will be called the
\emph{taut perimeter} of the spanning triangle (not to be confused with
the actual perimeter). 
\end{definition}

\begin{remark}\label{triangleremarks}
From Figure \ref{fig:lengths} and \eqref{metrics} we can conclude: 
\begin{enumerate}
\item The side from $a$ to $b$ has length at most\/ $4
  d_0(a,b)$. Similarly, the other two sides have lengths at most\/
  $4d_0(b,c)$ and\/ $4d_0(c,a)$. 
\item Each of the seven regions has perimeter at most $4 P$, where $P$ is
  the taut perimeter of the spanning triangle. 
\end{enumerate}
\end{remark}

\begin{remark}\label{fillingremark}
Each of the subcomplexes labelling a region in Figure \ref{fig:triangle}
is combinatorially isomorphic to one of the (sliced) cell complexes $X_i
\times X_j$, $L_i \times X_j$, or $L_i \times L_j$, by the Embedding
Lemma. All of these have Dehn functions that are $\preceq
n^{\alpha}$ (here we use the assumption that $\alpha \geq 2$). Let $C$ be
chosen so that $f(n) = Cn^{\alpha}$ is an upper bound for these Dehn
functions. Then, by Remark 
\ref{triangleremarks}(2), the triangle has filling area at most $7 C
(4P)^{\alpha}$ in $[X]_0$, where $P$ is the taut perimeter. 
\end{remark}

\begin{remark}\label{commonsides}
The definition of the path along any side of the triangle depends only on
its endpoints, and a choice of direction along the side (for the
non-canonical path in the middle segment). Given two triples of vertices
$a, b, c$ and $a, b, c'$, spanning triangles for both can be made to
agree along their sides from $a$ to $b$, by choosing the same
direction on those sides. 

A path along a side of a spanning triangle will be called a
\emph{spanning path}. 
\end{remark}

\subsection*{Short spanning paths}
Consider the spanning path from $a$ to $b$ in Figure
\ref{fig:triangle}, and suppose that $a$ and $b$ have distance at most
$1$ in $[X]_0$. If $a=b$ then the spanning path is a constant
path, of length $0$. If $a\not= b$, the geodesic from $a$ to $b$ is a
single edge in $[X]_0$, and we need to examine how this path may differ
from the spanning path. 

\begin{lemma}\label{bigons}
If $a$ and $b$ have distance $1$ in $[X]_0$ then the spanning path
from $a$ to $b$ and the geodesic edge from $a$ to $b$ together form a
loop with filling area at most $4$ in $[X]_0$. 
\end{lemma}

\begin{proof}
Since their distance is $1$, the points $a$ and $b$ differ in either
one or two coordinates. If they differ in only one coordinate, then
one finds that two of the three segments making up the spanning path
are constant paths. (There are three cases, according to the
coordinate where $a_i \not= b_i$.) The remaining segment has length
$1$, as shown by Figure \ref{fig:lengths}. Thus, up to
reparametrization, the spanning path agrees with the geodesic path 
and the filling area is $0$. 

Now suppose that $a$ and $b$ differ in two coordinates. Recall that
the spanning path joins the following points, in order: $(a_1, a_2,
a_3)$, $(a_1', b_2, a_3)$, $(a_1''', b_2, b_3)$, $(a_1, b_2', b_3)$,
and $(b_1, b_2, b_3)$. There are now three cases. If $a_1 = b_1$, then
one also finds that $a_1''' = a_1 = b_1$ and $b_2' = b_2$. The image
of the spanning path consists of two edges in the boundary of the cube
$[a_1', a_1] \times [a_2, b_2] \times [b_3, a_3] \subset X$, from
$(a_1, a_2, a_3)$ to $(a_1', b_2, a_3)$ to $(a_1, b_2, b_3)$. Together
with the geodesic edge, these edges are the boundary of a horizontal
$2$--cell in $[X]_0$, and the filling area is $1$. 

If $a_2 = b_2$ then one also has $a_1' = a_1$. Let $[a_1''', b_1]$
denote the length two path in $X_1$ from $a_1'''$ to $b_1$, with
midpoint $a_1$. Then $[a_1''', b_1] \times [b_2', a_2] \times [a_3,
b_3]$ is a union of two cubes in $X$. The image of the spanning
path consists of three edges on the boundary of these cubes. Together
with the geodesic edge, they form a quadrilateral that bounds two
triangles in $[X]_0$. 

If $a_3 = b_3$ then one also has $a_1' = a_1'''$. Let $[a_1', b_1]
\subset X_1$ and $[a_2, b_2'] \subset X_2$ denote the paths of length
$2$, with midpoints $a_1$ and $b_2$ respectively. The image of the
spanning path consists of three edges in $[a_1', b_1] \times [a_2,
b_2'] \times \{a_3\}$, and the geodesic edge also lies in this
subset. Choose a vertex $a_3' \in X_3$ such that 
$d_{X_3}(a_3, a_3') = 1$ and $(a_1, b_2, a_3') \in [X]_0$. Such a
vertex exists because $X_3$ is admissible. Then $[a_1', b_1] \times
[a_2, b_2'] \times [a_3, a_3'] \subset X$ is a union of four cubes 
in which the spanning path and geodesic path bound a disk made of four
horizontal triangles, with common vertex $(a_1, b_2, a_3')$. 
\end{proof}

\begin{proof}[Proof of Theorem \textup{\ref{mainthm}}]
Simple connectedness of $[X]_0$ will follow from 
the remainder of the proof, in which we construct disks in $[X]_0$
filling any given loop. 

Let $p$ be a closed edge path in $[X]_0$ of length $n > 3$. There is a
number $k \in \N$ such that $3 \cdot 2^{k-1} < n \leq 3 \cdot 2^{k}$. Let
$\hat{p}$ be a path of length $\hat{n} = 3 \cdot 2^{k}$ obtained by 
padding $p$ with steps that move distance $0$. Note that $p$ and
$\hat{p}$ have the same filling area. 

Now consider the triangulated disk $D$ shown in Figure
\ref{fig:farey}. It has $\hat{n}$ vertices along its boundary, $\hat{n}$
bigons around the outside, and $3 \cdot 2^{k} - 2$ triangles. Each
triangle has a \emph{depth}, where the central triangle has depth $0$,
its neighbors have depth $1$, and so on. For $i = 1, \dotsc, k$ there
are $3 \cdot 2^{i-1}$ triangles of depth $i$, and $k$ is the maximum
depth that occurs. 

\begin{figure}[ht]
\labellist
\hair 3pt
\Small
\pinlabel {$1$} [b] at 72.5 154
\pinlabel {$2$} [tr] at 10 47
\pinlabel {$3$} [tl] at 136 47

\pinlabel {$1$} [t] at 73 8
\pinlabel {$2$} [bl] at 136 114
\pinlabel {$3$} [br] at 10 114

\pinlabel {$3$} [bl] at 108 143
\pinlabel {$1$} [l] at 144.5 81
\pinlabel {$2$} [tl] at 108 18
\pinlabel {$3$} [tr] at 38.5 18
\pinlabel {$1$} [r] at 1 81
\pinlabel {$2$} [br] at 38.5 143.5

\endlabellist
\includegraphics{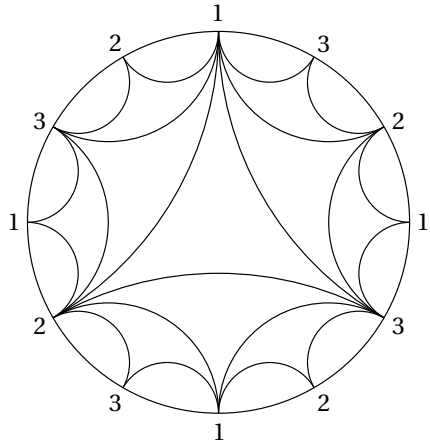}
\caption{Filling the loop with spanning triangles. A $3$--coloring of the
vertices is shown.}\label{fig:farey} 
\end{figure}

There is a $3$--coloring of the vertices of $D$: each vertex $v$ can be 
assigned a coordinate $\kappa(v) \in \{1, 2, 3\}$ such that $\kappa(v)
\not= \kappa(u)$ whenever $u,v$ bound an edge in $D$. Now identify the
boundary of $D$ with the path $\hat{p}$. Each vertex of $D$ is identified
with a vertex in $[X]_0$ and has coordinates in $X$. Writing $v = (v_1,
v_2, v_3)$, choose a monotone line $L_v$ in $X_{\kappa(v)}$ which
contains the point $v_{\kappa(v)}$. 

Each triangle in $D$ can now be filled with a spanning triangle for its
vertices, using the three vertices and the three monotone lines 
chosen for those vertices. The $3$--coloring ensures that this data
conforms to the requirements of the starting data for spanning
triangles. We start by filling the central triangle, and then proceed
to fill triangles in order of depth. Each new triangle to be filled
meets the previously filled triangles in a single edge, so 
by Remark \ref{commonsides}, the spanning triangle can be chosen to
use the same spanning path for that edge. Then the spanning
triangles fit together to yield a filling of $D$, minus the bigons. 

Declare the \emph{depth} of an interior edge in $D$ to be the minimum of
the depths of its neighboring triangles. Note that an edge of depth $i$
joins points on the boundary that bound a boundary arc of length
$2^{k-i}$, and so these points have distance at most $2^{k-i}$ in
$[X]_0$. 

The central triangle has taut perimeter at most $\hat{n}$, and the taut
perimeter of a depth $i$ spanning triangle ($i = 1, \dotsc, k$) is at
most $2^{k - (i-1)} + 2^{k-i} + 2^{k-i} =  2^{k-i+2}$. By Remark
\ref{fillingremark} the central spanning triangle has area at most $7C
(4\hat{n})^{\alpha} = 7 \cdot 12^{\alpha} C \cdot 2^{k\alpha}$ and a depth
$i$ spanning triangle has area at most $7 \cdot 4^{\alpha} C
(2^{k-i+2})^{\alpha}$. Now the total area of the spanning triangles is at
most 
\begin{align*}
7 \cdot 12^{\alpha} C \cdot 2^{k\alpha} + \sum_{i=1}^k 3 \cdot 2^{i-1} \cdot
7 \cdot 4^{\alpha} C (2^{k - i + 2})^{\alpha}
\ &= \ 7 \cdot 12^{\alpha} C \cdot 2^{k\alpha} + 21C \cdot
2^{k \alpha + 4 \alpha - 1} 
\sum_{i=1}^k 2^{(1-\alpha) i} \\ 
&< \ 7 \cdot 12^{\alpha} C \cdot 2^{k\alpha} + 21C \cdot 2^{k \alpha
  + 4 \alpha - 1} \\ 
&< \ 28C \cdot 12^{4\alpha} \cdot 2^{k\alpha}. 
\end{align*}
Each bigon has filling area at most $4$, by Lemma \ref{bigons}. Then
the filling area of $p$ is at most 
\begin{align*}
28C \cdot 12^{4\alpha} \cdot 2^{k\alpha} + 4 \cdot 3 \cdot 2^k
\ &< \ (28C + 1) \cdot 12^{4\alpha} \cdot 2^{k\alpha} \\
&= \ (28C+1) \cdot 12^{4\alpha} \cdot 2^{\alpha} \cdot 3^{-\alpha} \cdot (3
\cdot 2^{k-1})^{\alpha} \\
&< \ ((28C+1) \cdot 12^{4\alpha} \cdot 2^{\alpha} \cdot 3^{-\alpha})
n^{\alpha} 
\end{align*}
since $3 \cdot 2^{k-1} < n$. Therefore $\delta_{[X]_0}(n) \leq
Kn^{\alpha}$ where $K = (28C+1) \cdot 12^{4\alpha} \cdot 2^{\alpha} \cdot
3^{-\alpha}$. 
\end{proof}


\def\cprime{$'$}

\end{document}